\documentclass[usenames,dvipsnames]{amsart}

\usepackage{amssymb}
\usepackage{amsmath}
\usepackage[all]{xy}
\usepackage{colonequals}
\usepackage[english]{babel}
\usepackage{tikz}
\usetikzlibrary{matrix}
\usetikzlibrary{calc,intersections}
\usepackage[shortlabels]{enumitem}
\usepackage{faktor}
\usepackage{nicefrac}
\usepackage[ruled]{algorithm2e}

\usepackage[font=small,labelfont=bf]{caption} 

\usepackage[hyphens]{url}
\usepackage{hyperref}
\usepackage[hyphenbreaks]{breakurl}

\makeatletter
\let\@wraptoccontribs\wraptoccontribs
\makeatother

\DeclareMathOperator{\Cls}{Cls} 
\DeclareMathOperator{\Ann}{Ann}
\DeclareMathOperator{\ord}{ord}

\DeclareMathOperator{\Cl}{Cl}
\DeclareMathOperator{\mass}{mass}
\def\Q{\mathbb{Q}}
\def\Z{\mathbb{Z}}

\newcommand{\cM}{{\mathcal M}}
\newcommand{\cQ}{{\mathcal Q}}
\newcommand{\fra}{{\mathfrak a}}
\newcommand{\fraka}{{\mathfrak a}}
\newcommand{\frakp}{{\mathfrak p}}
\newcommand{\frakN}{{\mathfrak N}}

\newcommand{\frf}{{\mathfrak f}}

\newcommand{\p}{{\mathfrak p}}

\renewcommand{\bar}{\overline}

\newcommand{\set}[1]{\left\lbrace#1\right\rbrace }

\newcommand{\calO}{O}

\newcommand{\sO}{\mathsf{O}}
\newcommand{\sR}{\mathsf{R}}
\newcommand{\sL}{\mathsf{L}}
\newcommand{\OR}[1]{\sO_\sR\!\left(#1\right)}
\newcommand{\OL}[1]{\sO_\sL\!\left(#1\right)}
\newcommand{\ccR}[2]{\left(#1:#2\right)_\sR}
\newcommand{\ccL}[2]{\left(#1:#2\right)_\sL}
\newcommand{\simR}{\sim_\sR}

\newcommand{\dual}[1]{#1^{\sharp}}
\DeclareMathOperator{\trd}{trd}
\DeclareMathOperator{\nrd}{nrd}
\DeclareMathOperator{\discrd}{discrd}
\DeclareMathOperator{\condi}{(SEM)}
\DeclareMathOperator{\condii}{(LPP)}

\newtheorem{thm}{Theorem}[section] 
\newtheorem{lemma}[thm]{Lemma}     
\newtheorem{cor}[thm]{Corollary}
\newtheorem{prop}[thm]{Proposition}
\newtheorem*{mainresult}{Main Result}
\newtheorem{df}[thm]{Definition}
\newtheorem{remark}[thm]{Remark}
\newtheorem{example}[thm]{Example}

\author{Stefano Marseglia}
\address{Mathematical Institute, Utrecht University, The Netherlands}
\email{s.marseglia@uu.nl}
\author{Harry Smit}
\address{Department of Mathematics, University of Pennsylvania, United States of America}
\email{hsmit@sas.upenn.edu}
\contrib[\break with an appendix by]{John Voight}
\address{Department of Mathematics, Dartmouth College, 6188 Kemeny Hall, Hanover, NH 03755, USA}
\email{jvoight@gmail.com}

\title{Ideal classes of orders in quaternion algebras}

\begin{document}

\begin{abstract}
    We provide an algorithm that, given any order $O$ in a quaternion algebra over a global field, computes representatives of all right equivalence classes of right $O$-ideals, including the non-invertible ones.
    The theory is developed for a more general kind of algebras.
\end{abstract}

\maketitle


\section{Introduction}

Due to the unpredictable nature of ideal classes, algorithms to compute them have been actively sought after.
Arguably the most extensively studied case is that of the ideal classes of a maximal order in a number field.
These classes form an abelian group, known as the class group of the number field.
Even though many questions are still unanswered, efficient algorithms to compute them have been developed.
From the class group, thanks to the results contained in~\cite{KluenersPauli05}, we can efficiently deduce the group of classes of invertible ideals for any order in a number field.

In this paper we leave the relatively comfortable world of number fields for the one of finite-dimensional algebras over a field.
Quaternion algebras are a natural first stop. Here several important steps in the direction of understanding ideal classes have already been made.
Eichler orders, that is, intersections of two maximal orders, in a quaternion algebra over a number field, were the first to be considered, see~\cite{KirschmerVoight10}.
With restrictions on the number field and the quaternion algebras, an algorithm
to compute ideal classes for Bass orders, a family of orders that includes the Eichler orders, was developed in~\cite{PacettiSirolli14}.
Note that for a Bass order every right ideal is invertible (although possibly over a bigger order).
Finally, the algorithms contained in~\cite{KirschmerVoight10} 
are generalized in Appendix \ref{appendix:invidls},
in fact completely solving the problem of computing invertible right ideal classes for any order in a quaternion algebra.

Computations of such ideal classes have applications in several areas of number theory and arithmetic geometry.
In particular, we highlight the following two applications:
\begin{itemize}
    \item modular forms. Through Brandt matrices, ideal classes of orders in quaternion algebras produce modular forms. This was turned into an algorithm by Pizer~\cite{Pizer80}, and then generalized by several others: see for example Kohel~\cite{Kohel01}, Socrates--Whitehouse~\cite{SocratesWhitehouse05}, Demb\'el\'e~\cite{Dembele07} and Demb\'el\'e--Donnelly~\cite{DembeleDonnelly08}.
    For a survey, see Demb\'el\'e--Voight~\cite{DembeleVoight13}.
    \item abelian varieties over finite fields.
    The first author in~\cite{MarAV} and Oswal--Shankar in~\cite{OswalShankar} use both invertible and non-invertible ideal classes of orders in (products of) number fields to compute isomorphism classes of certain abelian varieties.
    In~\cite{XueYuSuperSpecial1},~\cite{XueYuSuperSpecial2} and~\cite{XueYuSuperSpecial3}, Xue--Yang--Yu and Xue--Yu--Zheng compute isomorphism classes of superspecial abelian surfaces using ideal classes of orders in quaternion algebras.
\end{itemize}

So far, the first application only makes use of the invertible ideal classes,
while the second gives a reason to compute the isomorphism classes of non-invertible ideals as well.
Such ideals are much harder to understand: they are not locally principal, and their existence is connected to the (bad) singularities of the order.
Xue--Yang--Yu and Xue--Yu--Zheng encounter only orders that are either Bass or close enough to being Bass that they can
produce a formula for the number of ideal classes.
On the other hand, if one wants to generalize this approach to higher-dimensional abelian varieties with quaternionic endomorphism algebras, one will presumably have to deal with orders where an algorithmic approach is worthwhile. In fact, this situation arises already for the commutative case of the first author and Oswal-Shankar.
In the commutative case, algorithms to compute all ideal classes are provided in~\cite{MarICM20} and improved in~\cite{MarCMType22_preprint}.

In this paper we deal with the non-commutative setting.
For missing definitions we refer the reader to the beginning of Section~\ref{sec:lattices}.
The main contribution of this paper is the following result.
\begin{mainresult}\label{mainresult}
    Let $R$ be a Dedekind domain whose field of fractions is a global field $F$, and
    let $O$ be an $R$-order in a quaternion algebra over $F$. 
    We give an algorithm that
    computes representatives of all right equivalence classes of right $O$-ideals, including the non-invertible ones.
\end{mainresult}
It turns out that it is easier to understand non-invertible ideals if we replace the notion of right equivalence with a coarser version.
Informally speaking, we say that two ideals
 are \emph{weakly right equivalent} if they are equal up to multiplication by an invertible ideal from the left.
In other words, the notion of weak right equivalence trivializes the invertible part of an ideal.
For example, an ideal is invertible if and only if it is weakly right equivalent to its right order.

To compute all the right equivalence classes of $O$, we compute the weak right equivalence classes with given right order $O$ and then incorporate the information from the invertible right equivalence classes. See Theorem~\ref{thm:righteqclasses} for more details.

We now give an overview of the paper. Aiming for greater generality, in each section we assume only the hypotheses we actually need, gradually adding them to end up with a working algorithm.
We stress that even though before we talked about orders in quaternion algebras over global fields, these constitute only a special case of our theory.

For a finite-dimensional algebra $B$ over the fraction field of a Dedekind domain, we introduce two variants of the notion of right equivalence in Section~\ref{sec:lattices}: weak right equivalence and local right equivalence. See Definitions~\ref{def:weakeq} and~\ref{def:loceq}, respectively.


In Section~\ref{sec:comparison} we compare local right equivalence and weak right equivalence.
In general, these two notions do not coincide as we show in Example~\ref{ex:kap}.
This difference essentially arises because the notion of invertibility and local principality are not always equivalent.
See Proposition~\ref{prop:loc_eq_same_wek_eq}.
In the same section we furthermore introduce two technical conditions $\condi$ and $\condii$, see Definitions~\ref{df:condi} and~\ref{df:condii}, that guarantee correctness of our algorithm to compute the weak right equivalence classes.
We remark that such conditions are satisfied if $B$ is commutative or has a standard involution, for example if $B$ is a quaternion algebra.

In Section~\ref{sec:duality}, under the hypothesis that our algebra $B$ is separable, together with $\condi$ and $\condii$, we give a method to compute the weak right equivalence classes.
Under some mild finiteness hypotheses, satisfied for example if $B$ is defined over a global field, in Section~\ref{sec:algorithms} we turn this method into an actual algorithm, and provide the pseudocode.
In the same section, we describe an algorithm that, given representatives of the weak right equivalence classes and a method to compute the right equivalence classes of invertible ideals, returns the right equivalence classes of all ideals, both invertible and non-invertible.
Section \ref{sec:brandtmatrices} contains an example: we use the algorithms to compute (extensions of) Brandt matrices and obtain some modular forms.
The implementation in Magma~\cite{Magma} of the algorithm, and code to reproduce the examples are available at
 \url{https://github.com/harryjustussmit/IdlClQuat}.

\subsection*{Acknowledgments}
The first author is supported by NWO grant VI.Veni.202.107.
The second author would like to thank the Max Planck Institute for Mathematics in Bonn.
The authors express their gratitude to Markus Kirschmer and John Voight for useful comments on a preliminary version of the paper. 

\section{Lattices\label{sec:lattices}}
Let $R$ be a Dedekind domain with fraction field $F$.
Let $B$ a finite-dimensional $F$-algebra. An $R$-{\bf lattice} in~$B$ is a finitely generated sub-$R$-module of $B$ that contains a basis of $B$. An $R$-{\bf order} of $B$ is an $R$-lattice that is also a subring (containing~$1$). When no confusion on the Dedekind domain $R$ can arise, we will simplify the terminology and write lattice and order.

Let $I$ and $J$ be lattices.
We say that $I$ is {\bf right equivalent} to $J$, and write $I \simR J$, if there exists $\alpha \in B^\times $ such that $I = \alpha J$.
The {\bf right colon} $\ccR{I}{J}$ is defined as $\set{x \in B : Jx \subseteq I}$. Similarly, the {\bf left colon } $\ccL{I}{J}$ is defined as $\set{x \in B : xJ \subseteq I}$.
The {\bf right order} $\OR{I}$ of $I$ is $\ccR{I}{I}$ and the {\bf left order} $\OL{I}$ of $I$ is $\ccL{I}{I}$. We say that $I$ is a {\bf right} (resp.~{\bf left}) $O$-{\bf ideal} if $O \subseteq \OR{I}$ (resp.~$O \subseteq \OL{I}$), and that $I$ is an $O$-$O'$-{\bf ideal} if $I$ is a right $O'$-ideal and a left $O$-ideal.

\begin{lemma} \label{lemma:colonalpha}
    Let $I$, $J$, and $K$ be lattices and $\alpha,\beta \in B^\times$.
    Then
    \begin{enumerate}[\normalfont(i)]
        \item \label{lemma:colonalpha:0}
            $\ccL{\alpha I}{\beta J} = \alpha \ccL{I}{J} \beta^{-1}$ and $\ccR{I\alpha}{J\beta} = \beta^{-1}\ccR{I}{J}\alpha$.
        \item \label{lemma:colonalpha:2}
             $\ccL{I}{J\alpha} = \ccL{I\alpha^{-1}}{J}$ and $\ccR{I}{\alpha J} = \ccR{\alpha^{-1}I}{J}$.
        \item \label{lemma:colonalpha:3}
            $\ccL{\ccL{I}{J}}{K} = \ccL{I}{KJ}$ and $\ccR{\ccR{I}{J}}{K} = \ccR{I}{JK}$.
        \item \label{lemma:colonalpha:5} $\ccL{\ccR{I}{J}}{K} = \ccR{\ccL{I}{K}}{J}$.
        \item \label{lemma:colonalpha:6}
            $I\ccL{J}{K} \subseteq \ccL{IJ}{K}$ and $\ccR{J}{K}I\subseteq \ccR{JI}{K}$.
        \item \label{lemma:colonalpha:7} $\ccL{I}{J}$ is an $\OL{I}$-$\OL{J}$-ideal, and $\ccR{I}{J}$ is an $\OR{J}$-$\OR{I}$-ideal.
    \end{enumerate}
\end{lemma}
\begin{proof}
    We will only prove one side for the two-sided parts of the lemma.
    By definition we have
    \begin{align*}
        \ccL{\alpha I}{\beta J} =
        \set{ x \in B : x\beta J\subseteq \alpha I } &=
        \set{ x \in B : \alpha^{-1}x\beta J\subseteq I } \\&=
        \set{ \alpha y \beta^{-1} \in B : yJ\subseteq I } =
        \alpha \ccL{I}{J} \beta^{-1},
    \end{align*}
    which proves Part~\ref{lemma:colonalpha:0}.
    For Part~\ref{lemma:colonalpha:2} we observe that
    \[
        \ccL{I}{J \alpha} =
        \set{ x \in B : x J \alpha \subseteq I } =
        \set{ x \in B : x J \subseteq I \alpha^{-1} } =
        \ccR{I \alpha^{-1}}{J}.
    \]
    Part~\ref{lemma:colonalpha:3} follows from
    \[
      \ccL{\ccL{I}{J}}{K} = \set{ x \in B : x K \subseteq \ccL{I}{J} } = \set{ x \in B : x K J \subseteq I } =  \ccL{I}{KJ}.
    \]
    For part~\ref{lemma:colonalpha:5}, observe that
    \[
    \ccL{\ccR{I}{J}}{K} = \set{ x \in B : JxK \subseteq I } = \ccR{\ccL{I}{K}}{J}.
    \]
    For part~\ref{lemma:colonalpha:6} note that $I\ccL{J}{K}K \subseteq IJ$.
    For part~\ref{lemma:colonalpha:7} take $x\in \ccL{I}{J}$, $y\in\OL{J}$ and $x\in \OL{I}$.
    Then
    \[ zxy J \subseteq zxJ \subseteq zI \subseteq I. \qedhere\]
\end{proof}

Let $\p$ be a maximal ideal of $R$ and denote by $R_\p$ the localization of $R$ at $\p$. Note that the fraction field of $R_{\p}$ is $F$. The {\bf localization} of an $R$-lattice $I$ at $\p$ is the $R_\p$-lattice $IR_\p$ (in $B$), and we denote the localization by $I_\p$. 
Localization commutes with addition, multiplication, and taking the colon of lattices.
Furthermore, two lattices are equal if and only if their localizations are equal, see for example~\cite[Thm.~9.4.9]{JV}.

An $R$-lattice $I$ is {\bf principal} if $I=\alpha\OR{I}$ for some $\alpha \in B$, or, equivalently, $I=\OL{I}\alpha$.
We say $I$ is {\bf locally principal} if for every maximal ideal $\p$ of $R$ the $R_{\p}$-lattice $I_{\p}$ is principal.

\begin{lemma}\label{lem:pullback_lattices}
    Let $O \subseteq O'$ be orders.
    If $J$ is a locally principal $R$-lattice with $\OR{J} = O'$ then there exists a locally principal $R$-lattice $I$ with $\OR{I} = O$ such that $IO' = J$.
\end{lemma}
\begin{proof}
    For every maximal ideal $\p$ there is an element $\alpha_{\p} \in B$ such that $J_{\p} = \alpha_{\p} O'_{\p}$. By~\cite[Thm.~9.4.9]{JV} we can choose $\alpha_{\p} = 1$ for all but finitely many $\p$.

    Again by~\cite[Thm.~9.4.9]{JV}, there exists an $R$-lattice $I$ such that $I_{\p} = \alpha_{\p} O_{\p}$.
    For any maximal ideal $\p$ we have $(IO')_{\p} = I_{\p} O'_{\p} = \alpha_{\p} O'_{\p} = J_{\p}$, hence $IO' = J$.
    Moreover, $\OR{I}_{\p} = \OR{I_{\p}} = O_{\p}$, thus $\OR{I} = O$.
\end{proof}

Two lattices $I$ and $J$ are called {\bf compatible} if $\OR{I}=\OL{J}$.
We say the lattice $I$ is {\bf right invertible} if there exists a lattice $I'$ such that $II'=\OL{I}$ and $I$ and $I'$ are compatible.
Similarly, $I$ is {\bf left invertible} if there exists a lattice $I'$ such that $I'I=\OR{I}$ and $I'$ and $I$ are compatible.
We say that $I$ is {\bf invertible} if it is both left and right invertible.

For a lattice $I$ we define the {\bf quasi-inverse} $I^{-1}$ as
\[ I^{-1}:=\set{ x \in B : I x I \subseteq I }. \]
Taking the quasi-inverse commutes with localization.
\begin{lemma}\label{lemma:char_qinv}
  For any lattice $I$ the following hold.
    \begin{enumerate}[\normalfont(i)]
        \item \label{lemma:char_qinv:1} $I^{-1} = \ccL{\OR{I}}{I}=\ccR{\OL{I}}{I}$.
        \item \label{lemma:char_qinv:2} $\OL{I} \subseteq \OR{I^{-1}}$ and $\OR{I} \subseteq \OL{I^{-1}}$.
    \end{enumerate}
\end{lemma}
\begin{proof}
    Both follow from the definition.
\end{proof}
Notice that the compatibility of the product $II^{-1}$ is not equivalent to the compatibility of $I^{-1}I$, as we show in Example~\ref{ex:sidedinvproj4}.

The next lemma is a small extension of \cite[Prop. 16.7.4]{JV} by showing that if the inverse of a lattice exists, it is unique.
\begin{lemma}\label{lemma:uniquesidedinv}
    Let $I$ be a lattice. Then
    if $I$ is right (resp.~left) invertible then the right (resp.~left) inverse is unique and equal to $I^{-1}$.
\end{lemma}
\begin{proof}
    Assume that $I$ is right invertible, that is, there exists a lattice $I'$ such that $II'$ is compatible and $II'=\OL{I}$.

    By Lemma~\ref{lemma:char_qinv}.\ref{lemma:char_qinv:1} and the equality $\OR{I}=\OL{I'}$ we have
    \[ I^{-1}I = \ccL{\OR{I}}{I}I  = \ccL{\OL{I'}}{I}I\subseteq \OL{I'}. \]
    Now, multiplying by $I'$ on the right, we obtain
    \[ I^{-1}II' \subseteq \OL{I'}I' = I', \]
    which leads to
    \[ I^{-1}\OL{I} \subseteq I'. \]
    Since clearly $I^{-1}\subseteq I^{-1}\OL{I}$, we see $I^{-1}\subseteq I'$.
    For the other inclusion, note that $II'=\OL{I}$ implies $I'\subseteq \ccR{\OL{I}}{I}$. Now $I'\subseteq I^{-1}$ follows from Lemma~\ref{lemma:char_qinv}.\ref{lemma:char_qinv:1}.
    The statement for left invertibility follows analogously.
\end{proof}

\begin{remark}
    Occasionally in the literature (for example in~\cite{Kap69}), a lattice $I$ is called left invertible if $I^{-1}I=\OR{I}$.
    In particular, the product is not required to be compatible.
    In Section~\ref{sec:comparison}, we call this weaker notion left projectivity, where we also show that if the algebra has a standard involution (as assumed in~\cite{Kap69}) then the two notions coincide (Proposition~\ref{prop:invC1C2}).
\end{remark}

The following lemma is a more general version of~\cite[Lem.~16.5.11]{JV}.
\begin{lemma}\label{lemma:orderinvmult}
    Let $I$ and $L$ be lattices.
    \begin{enumerate}[\normalfont(i)]
       \item \label{lemma:orderinvmultleft} If $L$ is right invertible and $\OL{L}\subseteq \OR{I}$, then $\OL{I} = \OL{IL}$.
       \item \label{lemma:orderinvmultright} If $L$ is left invertible and $\OR{L}\subseteq \OL{I}$, then $\OR{I} = \OR{LI}$.
    \end{enumerate}
\end{lemma}
\begin{proof}
    We only prove Part~\ref{lemma:orderinvmultleft}.
    By definition of the left order we have $\OL{I}\subseteq \OL{IL}$.
    Note that the assumption $\OL{L}\subseteq \OR{I}$ gives us $I=I\OL{L}$.
    To prove the converse inclusion, let $x\in \OL{IL}$ and let $L'$ be the right inverse of $L$.
    We obtain
    \[ xI = xI\OL{L}=xILL' \subseteq ILL' = I \OL{L} = I. \]
    Hence $\OL{IL} \subseteq \OL{I}$, which concludes the proof.
\end{proof}

\begin{lemma}\label{lemma:compatibleproductinvertible}
    Let $I$ and $J$ be compatible lattices. Assume that $I$ is left invertible and $J$ is right invertible.

    \begin{enumerate}[\normalfont(i)]
        \item \label{lemma:compatibleproductinvertible:1} If $I$ is also right invertible, then $IJ$ is right invertible.
        \item \label{lemma:compatibleproductinvertible:2} If $J$ is also left invertible, then $IJ$ is left invertible.
        \item \label{lemma:compatibleproductinvertible:3} If $I$ and $J$ are invertible, then $IJ$ is invertible.
    \end{enumerate}
\end{lemma}
\begin{proof}
    Under the assumptions of Part~\ref{lemma:compatibleproductinvertible:1}, we have
    \[
        IJ (J^{-1}I^{-1}) = I \OL{J} I^{-1} = I \OR{I} I^{-1} = I I^{-1} = \OL{I} = \OL{IJ},
    \]
    where the first equality holds by right invertibility of $J$, the second by the compatibility of $I$ and $J$, the fourth one by right invertibility of $I$, and the last one by Lemma~\ref{lemma:orderinvmult}.\ref{lemma:orderinvmultleft}.
    We show that $IJ$ and $J^{-1}I^{-1}$ are compatible. This follows as $\OR{IJ} = \OR{J}$ by Lemma~\ref{lemma:orderinvmult}.\ref{lemma:orderinvmultright},
    $\OR{J} = \OL{J^{-1}}$ by Lemma~\ref{lemma:uniquesidedinv}, and $\OL{J^{-1}} = \OL{J^{-1}I^{-1}}$ by  Lemma~\ref{lemma:orderinvmult}.\ref{lemma:orderinvmultleft}. The conditions of the lemma are met as
    \[
    \OL{I^{-1}} = \OR{I} = \OL{J} \subseteq \OR{J^{-1}},
    \]
    where the first equality holds by Lemma~\ref{lemma:uniquesidedinv}, the second by compatibility of $I$ and $J$, and the inclusion by Lemma~\ref{lemma:char_qinv}.\ref{lemma:char_qinv:2}.

    For Part~\ref{lemma:compatibleproductinvertible:2}, we follow the same strategy, noting that
    \[
        (J^{-1}I^{-1})IJ = J^{-1} \OR{I} J  = J^{-1} \OL{J} J  = J^{-1} J = \OR{J} = \OR{IJ}.
    \]
    Compatibility of $J^{-1}I^{-1}$ and $IJ$ follows from
    \[
    \OR{J^{-1} I^{-1}} = \OR{I^{-1}} = \OL{I} = \OL{IJ}.
    \]
    Lastly, Part~\ref{lemma:compatibleproductinvertible:3} follows directly from Parts~\ref{lemma:compatibleproductinvertible:1} and~\ref{lemma:compatibleproductinvertible:2}.
\end{proof}

The next two lemmas deal with the interaction between invertible lattices and colons.
\begin{lemma}\label{lemma:coloninvnum}
    Let $I$, $J$ and $L$ be lattices with $L$ invertible and $\OR{L}\subseteq \OL{J}$.
    Then $\ccL{LJ}{I} = L\ccL{J}{I}$.
\end{lemma}
\begin{proof}
    Multiplying $\ccL{LJ}{I}I \subseteq LJ$ by $L^{-1}$ on the left leads to
    \[ L^{-1} \ccL{LJ}{I}I \subseteq L^{-1}LJ = \OR{L} J \subseteq \OL{J} J = J. \]
    We get $L^{-1} \ccL{LJ}{I} \subseteq \ccL{J}{I}$, which multiplied on the left by $L$ gives the inclusion
    \[ L L^{-1} \ccL{LJ}{I} \subseteq L \ccL{J}{I}. \]
    However,
    \[
    L L^{-1} \ccL{LJ}{I} = \OL{L}\ccL{LJ}{I} = \ccL{LJ}{I},
    \]
    as the first equality follows from right invertibility of $L$ and the second follows from $\OL{L} \subseteq \OL{LJ}$ and Lemma~\ref{lemma:colonalpha}.\ref{lemma:colonalpha:7}. We conclude that $\ccL{LJ}{I} \subseteq  L \ccL{J}{I}$. Equality follows from Lemma~\ref{lemma:colonalpha}.\ref{lemma:colonalpha:6}.
\end{proof}
\begin{lemma}\label{lemma:coloninvden}
    Let $I$, $J$ and $L$ be lattices with $L$ invertible and $\OR{L}\subseteq \OL{I}$.
    Then $\ccL{J}{LI} = \ccL{J}{I}L^{-1}$.
\end{lemma}
\begin{proof}
    Observe that
    \[ \ccL{J}{I}L^{-1} LI = \ccL{J}{I}\OR{L}I \subseteq \ccL{J}{I}\OL{I} I = \ccL{J}{I}I \subseteq J, \]
    which tells us that $\ccL{J}{I}L^{-1} \subseteq \ccL{J}{LI}$.

    As $\ccL{J}{LI}LI \subseteq J$, we have $\ccL{J}{LI}L \subseteq \ccL{J}{I}$. Multiplying this with $L^{-1}$ on the right gives the last inclusion in
    \[ \ccL{J}{LI} = \ccL{J}{LI}\OL{L}=\ccL{J}{LI}LL^{-1} \subseteq \ccL{J}{I}L^{-1}. \]
    The first equality follows from Lemma~\ref{lemma:colonalpha}.\ref{lemma:colonalpha:7}, since $\OL{L} \subseteq\OL{LI}$.
\end{proof}

We conclude this section with a lemma on the finiteness of the number of intermediate lattices, which is required for our algorithms to terminate.


\begin{lemma}\label{lem:finmanylatt}
    Consider the following statements.
    \begin{enumerate}[\normalfont(i)]
        \item\label{lem:finmanylatt:1} The Dedekind domain $R$ is residually finite, that is, for every maximal ideal $\p$ of $R$ the quotient ring $R/\p$ is finite.
        \item\label{lem:finmanylatt:2} For every inclusion of $R$-lattices $L\subseteq L'$ the quotient $L'/L$ is finite.
        \item\label{lem:finmanylatt:3} For every inclusion of $R$-lattices $L\subseteq L'$ there are only finitely many $R$-lattices $L_0$ such that $L\subseteq L_0\subseteq L'$.
    \end{enumerate}
    Statement~\ref{lem:finmanylatt:1} is equivalent to~\ref{lem:finmanylatt:2} and either implies~\ref{lem:finmanylatt:3}. If we assume, additionally, that $\dim_F B \geq 2$, then all three statements are equivalent.
\end{lemma}
\begin{proof}
    Assume~\ref{lem:finmanylatt:1} holds. Let $L\subseteq L'$ be two lattices. Consider the finitely generated $R$-module $M = L'/L$. Observe that $M$ is torsion by~\cite[Lem.~9.3.5.(b)]{JV}, hence there exist nonzero ideals $I_1, \dots, I_n$ such that
    \[
        M \simeq \bigoplus_{i=1}^n R/I_i.
    \]
    Since every $I_i$ can be factored into a product of maximal ideals, every $R/I_i$ is finite, and thus $M = L'/L$ is also finite, which shows~\ref{lem:finmanylatt:2}. As the lattices $L_0$ between $L$ and $L'$ are in bijection with the sub-$R$-modules of $M$, there are finitely many choices for $L_0$, hence~\ref{lem:finmanylatt:3} holds.
    Now, assume that~\ref{lem:finmanylatt:1} does not hold. Let $\p$ be a maximal ideal with infinite residue field $k=R/\p$.
    Pick any order $O$ in $B$ and consider the inclusion of lattices $\p O\subset O$. The quotient $V=O/\p O$ is a non-trivial $k$-vector space. In particular, it is infinite as $k$ is infinite, hence~\ref{lem:finmanylatt:2} does not hold. If we also assume that $\dim_F B \geq 2$, since $\dim_k V = \dim_F B$, $V$ contains infinitely many sub-$k$-vector spaces. As these are in bijection with the intermediate lattices $\p O\subseteq L_0\subseteq O$, the result follows.
\end{proof}

\section{Weak and local equivalence}\label{sec:weak_and_local}

In this section we weaken right equivalence of lattices in two ways; \emph{weak right} and \emph{local right} equivalence. Informally speaking, two lattices are weakly right equivalent if they are equal up to multiplication with an invertible ideal from the left. In the next section we compare these two notions and find conditions under which they are equivalent, and show how they can be used to compute non-invertible right equivalence classes. Recall our notation from the previous section; let $R$ be a Dedekind domain with fraction field $F$, let $B$ be a finite-dimensional $F$-algebra and let $I$ and $J$ be $R$-lattices.

\subsection{Weak equivalence}

\begin{df}\label{def:weakeq}
    We say that lattices $I$ and $J$ are {\bf weakly right equivalent} if there exists an invertible lattice $L$ with $\OR{L} \subseteq \OL{J}$ such that $I = LJ$.
\end{df}
As the name suggests, being weakly right equivalent is an equivalence relation.
This is not obvious and will be proved later in Proposition~\ref{prop:wk_eq_is_eq_rel}.

\begin{thm}\label{thm:wk_eq}
    The following are equivalent:
    \begin{enumerate}[\normalfont(i)]
        \item \label{thm:wk_eq:a}
        $I$ and $J$ are weakly right equivalent.
        \item \label{thm:wk_eq:e}
        There exists an invertible lattice $L'$ with $\OR{L'} \subseteq \OL{I}$ such that $L'I=J$.
        \item \label{thm:wk_eq:b}
        $\ccL{I}{J}$ is invertible with inverse $\ccL{J}{I}$.
    \end{enumerate}
\end{thm}
\begin{proof}
    We prove that~\ref{thm:wk_eq:a} implies~\ref{thm:wk_eq:e}.
    So, there exists an invertible lattice $L$ with
    $\OR{L}\subseteq \OL{J}$ such that $I = LJ$.
    Observe that by Lemma~\ref{lemma:uniquesidedinv} we have
    \[ LL^{-1} = \OL{L} = \OR{L^{-1}} \text{ and } L^{-1}L=\OR{L}=\OL{L^{-1}}. \]
    Hence multiplying the relation $I=LJ$ with $L^{-1}$ on the left gives us
    \[ L^{-1}I=L^{-1} L J = \OR{L} J = J, \]
    since $\OR{L} \subseteq \OL{J}$.
    Moreover, we see that
    \[ \OR{L^{-1}} = \OL{L} \subseteq \OL{LJ} = \OL{I}. \]
    If we take $L'=L^{-1}$, we then see that we have proven~\ref{thm:wk_eq:e}.
    The converse holds by symmetry.

    We turn to the implication~\ref{thm:wk_eq:a} $\Longrightarrow$~\ref{thm:wk_eq:b}. By Lemma~\ref{lemma:coloninvnum} we get
    \begin{equation}\label{eq:cIJisLOLJ}
        \ccL{I}{J}=\ccL{LJ}{J}=L\OL{J}
    \end{equation}
    and by Lemma~\ref{lemma:coloninvden} we have
    \begin{equation}\label{eq:cJIisOLJL-1}
         \ccL{J}{I}=\ccL{J}{LJ}=\OL{J}L^{-1}.
    \end{equation}
    Therefore
    \[\ccL{J}{I}\ccL{I}{J}=\OL{J}L^{-1} L\OL{J} = \OL{J}\OR{L}\OL{J}=\OL{J}.\]
    By the first part of the proof we also have that $J=L^{-1}I$.
    Since, as we observed before that $\OR{L^{-1}}\subseteq \OL{I}$, we obtain from Lemmas~\ref{lemma:coloninvnum} and~\ref{lemma:coloninvden} that
    \[ \ccL{I}{J} = \ccL{I}{L^{-1}I}=\OL{I}L \]
    and
    \[ \ccL{J}{I} = \ccL{L^{-1}I}{I} = L^{-1}\OL{I}. \]
    It follows that
    \[ \ccL{I}{J}\ccL{J}{I}=\OL{I}LL^{-1}\OL{I}=\OL{I}\OR{L^{-1}}\OL{I}=\OL{I}. \]

    By Equation \eqref{eq:cIJisLOLJ} and Lemma~\ref{lemma:orderinvmult}.\ref{lemma:orderinvmultright} we have that
    \begin{equation}\label{eq:wk_eq:compolJ1}
        \OR{\ccL{I}{J}} = \OR{L\OL{J}} = \OR{\OL{J}}=\OL{J}.
    \end{equation}
    Also, we obtain from Equation \eqref{eq:cJIisOLJL-1} and Lemma~\ref{lemma:orderinvmult}.\ref{lemma:orderinvmultleft} that
    \begin{equation}\label{eq:wk_eq:compolJ2}
        \OL{\ccL{J}{I}} = \OL{ \OL{J} L^{-1} } = \OL{\OL{J}} = \OL{J}.
    \end{equation}
    So the lattices $\ccL{I}{J}$ and $\ccL{J}{I}$ are compatible.
    Similarly, since $\OL{L}=\OR{L^{-1}}\subseteq \OL{I}$, using Lemma~\ref{lemma:orderinvmult}.\ref{lemma:orderinvmultleft}, we see
    \begin{equation}\label{eq:wk_eq:compolI1}
        \OL{\ccL{I}{J}} = \OL{ \OL{I}L }=\OL{I},
    \end{equation}
    and, using Lemma~\ref{lemma:orderinvmult}.\ref{lemma:orderinvmultright}, we obtain
    \begin{equation}\label{eq:wk_eq:compolI2}
        \OR{\ccL{J}{I}} = \OR{L^{-1}\OL{I}} = \OL{I}.
    \end{equation}
    Therefore also $\ccL{J}{I}$ and $\ccL{I}{J}$ are compatible.
    This concludes the proof that~\ref{thm:wk_eq:a} implies~\ref{thm:wk_eq:b}.

    Finally, assume that~\ref{thm:wk_eq:b} holds.
    Since $1\in\ccL{I}{J}\ccL{J}{I}$ we have the following inclusions
    \[ I \subseteq \ccL{I}{J}\ccL{J}{I} I \subseteq \ccL{I}{J} J \subseteq I, \]
    and therefore
    \begin{equation}\label{eq:wk_eq:IeqccLIJJ}
        I = \ccL{I}{J} J.
    \end{equation}
    By assumption, $\ccL{I}{J}$ is invertible with inverse $\ccL{J}{I}$.
    Hence
    \[ \OR{\ccL{I}{J}} = \ccL{J}{I}\ccL{I}{J} \subseteq \OL{J}. \]
    In particular, we obtain that $L=\ccL{I}{J}$ satisfies all the requirements to prove~\ref{thm:wk_eq:a}.
\end{proof}

 In the following proposition we collect some facts that were deduced in the proof of the previous theorem.

\begin{prop}\label{prop:wk_eq_cons}
    Let $I$ and $J$ be weakly right equivalent lattices. Then
    \begin{enumerate}[\normalfont(i)]
        \item \label{prop:wk_eq_cons:a} $\OR{I}=\OR{J}$.
        \item \label{prop:wk_eq_cons:b} $\OL{\ccL{J}{I}}=\OR{\ccL{I}{J}} = \OL{J}$.
        \item \label{prop:wk_eq_cons:c} $\OL{\ccL{I}{J}} = \OR{\ccL{J}{I}} =\OL{I}$.
        \item \label{prop:wk_eq_cons:f} there exists a unique invertible lattice $L$ such that $I=LJ$ and $\OR{L}=\OL{J}$, namely $L=\ccL{I}{J}$.
    \end{enumerate}
\end{prop}
\begin{proof}
    By assumption we have that $I=LJ$ for an invertible ideal $L$ with $\OR{L}\subseteq \OL{J}$.
    Hence Lemma~\ref{lemma:orderinvmult}.\ref{lemma:orderinvmultright} implies Part~\ref{prop:wk_eq_cons:a}.
    Part~\ref{prop:wk_eq_cons:b} follows from Equations \eqref{eq:wk_eq:compolJ1} and \eqref{eq:wk_eq:compolJ2}, while Part~\ref{prop:wk_eq_cons:c} follows from Equations \eqref{eq:wk_eq:compolI1} and \eqref{eq:wk_eq:compolI2}.
    For Part~\ref{prop:wk_eq_cons:f}, observe that $L=\ccL{I}{J}$ satisfies $I = LJ$ by  Equation \eqref{eq:wk_eq:IeqccLIJJ} and $\OR{L}=\OL{J}$ by Part~\ref{prop:wk_eq_cons:b}.
    Uniqueness follows from Equation \eqref{eq:cIJisLOLJ}.
\end{proof}

\begin{prop}\label{prop:wk_eq_is_eq_rel}
    Weak right equivalence is an equivalence relation.
\end{prop}
\begin{proof}
    Reflexivity is clear from the definition.
    Symmetry is contained in Theorem~\ref{thm:wk_eq}.
    We show the relation is transitive.
    Assume lattices $J_1$ and $J_2$ are weakly right equivalent, and $J_2$ and $J_3$ are weakly right equivalent.
    By Proposition~\ref{prop:wk_eq_cons}.\ref{prop:wk_eq_cons:f}, the colon lattices $\ccL{J_1}{J_2}$ and $\ccL{J_2}{J_3}$ are invertible, $J_1 = \ccL{J_1}{J_2} J_2$ and $J_2 = \ccL{J_2}{J_3} J_3$, and $\OR{\ccL{J_1}{J_2}} = \OL{J_2}$ and $\OR{\ccL{J_2}{J_3}} = \OL{J_3}$.

    Note that $\ccL{J_1}{J_2}$ and $\ccL{J_2}{J_3}$ are compatible, as from Proposition~\ref{prop:wk_eq_cons}.\ref{prop:wk_eq_cons:b} and~\ref{prop:wk_eq_cons:c} it follows that
    \[
        \OR{\ccL{J_1}{J_2}} = \OL{J_2} = \OL{\ccL{J_2}{J_3}}.
    \]
    As both $\ccL{J_1}{J_2}$ and $\ccL{J_2}{J_3}$ are invertible, the same is true for the lattice $\ccL{J_1}{J_2}\ccL{J_2}{J_3}$ by Lemma~\ref{lemma:compatibleproductinvertible}.\ref{lemma:compatibleproductinvertible:3}.

    Furthermore, Lemma~\ref{lemma:orderinvmult}.\ref{lemma:orderinvmultright} shows that
    \[
    \OR{\ccL{J_1}{J_2}\ccL{J_2}{J_3}} = \OR{\ccL{J_2}{J_3}} = \OL{J_3}.
    \]
    In view of the equality
    \[
    \ccL{J_1}{J_2}\ccL{J_2}{J_3} J_3 = \ccL{J_1}{J_2}J_2 = J_1,
    \]
    we conclude that $J_1$ and $J_3$ are weakly right equivalent.
\end{proof}

Invertibility of a lattice can be reformulated in terms of weak right equivalence as follows.
\begin{prop}\label{prop:invOR}
    A lattice $I$ is invertible if and only if $I$ is weakly right equivalent to $\OR{I}$.
\end{prop}
\begin{proof}
    Assume that $I$ is invertible. Then the equality $I=I\OR{I}$ shows that $I$ is weakly right equivalent to $\OR{I}$.
    Conversely, assume that $I$ is weakly equivalent to $\OR{I}$.
    Then by Proposition~\ref{prop:wk_eq_cons}.\ref{prop:wk_eq_cons:f}, we have $I=L\OR{I}$ with $L$ invertible and satisfying $\OR{L} = \OL{\OR{I}}=\OR{I}$. So in particular $I = L\OR{I} = L\OR{L} = L$ is invertible.
\end{proof}

In order to determine the right equivalence classes of lattices with prescribed right order, the weak right equivalence classes form a stepping stone. The bridge between the two notions is given by multiplication with invertible lattices, as made precise by the following theorem, where we describe the fibers of the natural surjection
\[ \set{\parbox{4cm}{\centering right equivalence classes with right order $S$}}
 \longrightarrow
\set{\parbox{4cm}{\centering weak right equivalence classes with right order $S$}}. \]

\begin{thm}\label{thm:righteqclasses}
    Let $S$ be an order in $B$.
    Let $\mathcal{J}$ be a set of representatives $J$ of the weak right equivalence classes satisfying $\OR{J}=S$.
    For each $J \in \mathcal{J}$, let $\mathcal{L}_J$ be a set of representatives $L$ of the right equivalence classes of invertible ideals with $\OR{L}=\OL{J}$.
    Let $I$ be a lattice with $\OR{I}=S$.
    Then there is a unique $J \in \mathcal{J}$ and a unique $L \in \mathcal{L}_J$ such that
    \[ I \simR L J. \]
\end{thm}
\begin{proof}
    By Proposition~\ref{prop:wk_eq_cons}.\ref{prop:wk_eq_cons:a}, we know that the right order of a lattice is an invariant of its weak right equivalence class.
    Hence, since weak right equivalence is an equivalence relation by Proposition~\ref{prop:wk_eq_is_eq_rel}, there exists a unique $J \in \mathcal{J}$ such that $I$ is weakly right equivalent to $J$.
    By Proposition~\ref{prop:wk_eq_cons}.\ref{prop:wk_eq_cons:f}, there exists a unique invertible lattice $L'$ with $\OR{L'}=\OL{J}$ such that $I=L'J$.
    Let $L \in \mathcal{L}_J$ be the representative such that $L'\simR L$.
    Then by construction we have $I\simR L J$.
\end{proof}

\subsection{Local equivalence}

As the name suggests, we define two lattices to be locally right equivalent if they are right equivalent locally at every maximal ideal of $R$.
\begin{df}\label{def:loceq}
    We say that lattices $I$ and $J$ are {\bf locally right equivalent} if~$I_{\p} \sim_R J_{\p}$ for every maximal ideal $\p$ of $R$.
\end{df}
It is clear from the definition that being locally right equivalent is an equivalence relation. As every locally principal lattice is invertible (see for example~\cite[Cor.~16.5.10]{JV}), one expects local right equivalence to be stronger than weak right equivalence. This is confirmed by the next proposition.

\begin{thm}\label{thm:loc_eq}
    The following are equivalent.
    \begin{enumerate}[\normalfont(i)]
        \item \label{thm:loc_eq:a}
        $I$ and $J$ are locally right equivalent.
        \item \label{thm:loc_eq:b}
        $\ccL{I}{J}$ is locally principal with inverse $\ccL{J}{I}$.
        \item \label{thm:loc_eq:c}
        There exists a locally principal lattice $L$ with $\OR{L} \subseteq \OL{J}$ such that $LJ=I$.
    \end{enumerate}
\end{thm}
\begin{proof}
    Let $\p$ be a maximal ideal of $R$. By assumption there is an $\alpha \in B^\times$ such that $I_\p=\alpha J_\p$.
    Then by Lemma~\ref{lemma:colonalpha}.\ref{lemma:colonalpha:0} we have
    \[ \big(\ccL{I}{J}\big)_\p = \ccL{I_\p}{J_\p} = \ccL{\alpha J_\p}{J_\p} = \alpha \OL{J_\p}.\]
    It follows that
    \[ \OR{\big(\ccL{I}{J}\big)_\p } = \OL{J_\p}.\]
    Combined, we deduce that $\big(\ccL{I}{J}\big)_\p$ is principal. Moreover, by symmetry we $\big(\ccL{J}{I}\big)_\p = \OL{J_\p} \alpha^{-1}$.
    We see that $\ccL{I}{J}$ is locally principal, and $\ccL{I}{J}$ and $\ccL{J}{I}$ are locally compatible inverses of each other, hence also globally.

  Assume now that~\ref{thm:loc_eq:b} holds. It follows from Theorem~\ref{thm:wk_eq} that $I$ and $J$ are weakly right equivalent, hence by Proposition~\ref{prop:wk_eq_cons}.\ref{prop:wk_eq_cons:f} the lattice $L=\ccL{I}{J}$ satisfies the requirements of~\ref{thm:loc_eq:c}.

  Finally, assume that~\ref{thm:loc_eq:c} holds. For any maximal ideal $\p$ of $R$ we have $\OR{L_\p} \subseteq \OL{J_\p}$, and there is an $\alpha \in B^\times$ such that $L_\p = \alpha \OR{L_\p}$. Therefore
  \[
  I_\p = L_\p J_\p = \alpha \OR{L_\p} J_{\p} = \alpha J_{\p},
  \]
  which gives~\ref{thm:loc_eq:a}.
\end{proof}

\begin{cor}\label{cor:locrx_wkrx}
Let $I$ and $J$ be locally right equivalent $R$-lattices. Then $I$ and $J$ are weakly right equivalent.
\end{cor}

\begin{proof}
Every locally principal lattice is invertible, hence this follows from Theorem~\ref{thm:loc_eq} and Theorem~\ref{thm:wk_eq}.
\end{proof}

As in the case of weak right equivalence, we enumerate some properties of locally right equivalent lattices.

\begin{prop}\label{prop:loc_eq_cons}
    For locally right equivalent lattices $I$ and $J$ the following hold.
    \begin{enumerate}[\normalfont(i)]
        \item \label{prop:loc_eq_cons:a} $\OR{I}=\OR{J}$.
        \item \label{prop:loc_eq_cons:b} $\OL{\ccL{J}{I}}=\OR{\ccL{I}{J}} = \OL{J}$.
        \item \label{prop:loc_eq_cons:c} $\OL{\ccL{I}{J}} = \OR{\ccL{J}{I}} =\OL{I}$.
        \item \label{prop:loc_eq_cons:f} There exists a unique locally principal lattice $L$ such that $I=LJ$ and $\OR{L}=\OL{J}$, namely $L=\ccL{I}{J}$.
    \end{enumerate}
\end{prop}
\begin{proof}
    Parts~\ref{prop:loc_eq_cons:a},~\ref{prop:loc_eq_cons:b}, and~\ref{prop:loc_eq_cons:c} follow from Corollary~\ref{cor:locrx_wkrx} and Proposition~\ref{prop:wk_eq_cons}.
    Finally, Part~\ref{prop:loc_eq_cons:f} follows from Proposition~\ref{prop:wk_eq_cons}.\ref{prop:wk_eq_cons:f}, noting that by Corollary~\ref{cor:locrx_wkrx} and Theorem~\ref{thm:loc_eq}, $\ccL{I}{J}$ is locally principal.
\end{proof}

\begin{prop}\label{prop:loc_pr_OR}
    Let $I$ be an $R$-lattice.
    Then $I$ is locally principal if and only if $I$ is locally right equivalent to $\OR{I}$.
\end{prop}
\begin{proof}
    Let $\p$ be a maximal ideal of $R$. Then $I_\p=\alpha \OR{I_\p} = \alpha \OR{I}_\p$ is equivalent to $I_\p \simR  \OR{I}_\p$.
\end{proof}

In the following theorem we describe the fibers of the natural surjection (cf.~Theorem~\ref{thm:righteqclasses}).
\[ \set{\parbox{4cm}{\centering right equivalence classes with right order $S$}}
 \longrightarrow
\set{\parbox{4cm}{\centering local right equivalence classes with right order $S$}}. \]
\begin{thm}\label{thm:rightloceqclasses}
    Let $S$ be an order in $B$.
    Let $\mathcal{J}$ be a set of representatives $J$ of the locally right equivalence classes satisfying $\OR{J}=S$.
    For each $J \in \mathcal{J}$ let $\mathcal{L}_J$ be a set of representatives $L$ of the right equivalence classes of locally principal ideals with $\OR{L}=\OL{J}$.
    Let $I$ be a lattice with $\OR{I}=S$.
    Then there is a unique $J \in \mathcal{J}$ and a unique $L \in \mathcal{L}_J$ such that
    \[ I \simR L J. \]
\end{thm}
\begin{proof}
    By Proposition~\ref{prop:loc_eq_cons}.\ref{prop:loc_eq_cons:a}, we know that the right order of a lattice is an invariant of its weak right equivalence class.
    Hence there exists a unique $J \in \mathcal{J}$ such that $I$ is locally right equivalent to $J$.
    By Proposition~\ref{prop:loc_eq_cons}.\ref{prop:loc_eq_cons:f}, there exists a unique locally principal lattice $L'$ with $\OR{L'}=\OL{J}$ such that $I=L'J$.
    Let $L \in \mathcal{L}_J$ be the representative such that $L'\simR L$.
    Then by construction we have $I \simR L J$.
\end{proof}

\section{Comparison between weak and local equivalence}\label{sec:comparison}
As usual, let $R$ be a Dedekind domain with fraction field $F$, and let $B$ be a finite-dimensional $F$-algebra. In the previous section we have introduced two notions of equivalence: local right equivalence and weak right equivalence. In Corollary~\ref{cor:locrx_wkrx} we showed that locally right equivalent lattices are also weakly right equivalent, but the converse is not true in general, as is shown by the next example.

\begin{example}[{\cite[p.~221]{Kap69}}]\label{ex:kap}
    Let $R$ be a discrete valuation ring with field of fractions $F$ and uniformizer $\pi$.
    Consider the lattices in $\cM_3(F)$ given by
    \[ I=
       \begin{pmatrix}
       \pi R &\pi R & R\\    \pi R &\pi R & R\\    R &R & R
       \end{pmatrix},
       \quad
       J=
       \begin{pmatrix}
       R & R & R\\ R & R & R\\    R & R & \pi R
       \end{pmatrix}.
    \]
    One computes that
    \[
        \OR{I}=\OL{J}=
        \begin{pmatrix}
        R & R & R\\    R & R & R\\    \pi R & \pi R & R
        \end{pmatrix},
        \quad
        \OL{I}=\OR{J}=
        \begin{pmatrix}
        R & R & \pi R\\    R & R & \pi R\\    R & R & R
        \end{pmatrix}.
    \]
    Hence $I$ is compatible with $J$ and $J$ is compatible with $I$. Furthermore
    \[ IJ=\OL{I},\quad JI=\OR{I},\]
    proving that $I$ is (two-sided) invertible with (two-sided) inverse $J = I^{-1}$. This means that $I$ is weakly right equivalent to $\OR{I}$ by Proposition~\ref{prop:invOR}.
    On the other hand it is clear that $I$ is not (locally) principal, thus Proposition~\ref{prop:loc_pr_OR} states that $I$ is not locally right equivalent to $\OR{I}$.
\end{example}

In Definitions~\ref{df:condi} and~\ref{df:condii} we introduce and discuss a pair of conditions on the algebra and its orders that guarantee that weak right equivalence coincides with local right equivalence. Noticably, commutative and quaternion algebras satisfy these conditions,
which is proven in Proposition~\ref{prop:invC1C2}.
\begin{df}\label{df:condi}
    We say that an inclusion of two orders $O \subseteq O'$ in $B$ satisfies~$\condi$ if
    for every invertible $R$-lattice $J$ satisfying $\OR{J} = O'$ there exists an invertible~$R$-lattice $I$ with $\OR{I} = O$ such that $IO' = J$. Additionally, the algebra~$B$ satisfies~$\condi$ if any inclusion of orders in~$B$ satisfies~$\condi$.
\end{df}
$\condi$ stands for "Surjective Extension Map".
We give two equivalent conditions that guarantee that certain inclusions of orders satisfy $\condi$.

\begin{prop}\label{prop:loc_eq_same_wek_eq}
    For an order $O'$ in $B$, the following are equivalent.
    \begin{enumerate}[\normalfont(i)]
        \item \label{prop:loc_eq_same_wek_eq:invlocprinc} A lattice $I$ with $\OR{I}=O'$ is invertible if and only if it is locally principal.
        \item \label{prop:loc_eq_same_wek_eq:locrx_wkrx} Two lattices $I$ and $J$ with $\OR{I}=O'$ in $B$ are locally right equivalent if and only if they are weakly right equivalent.
    \end{enumerate}
    Moreover, if any of the two statements is satisfied for $O'$ then any inclusion of orders $O \subseteq O'$ satisfies $\condi$.
\end{prop}
\begin{proof}
    Recall that locally principal lattices are invertible (\cite[Cor.~16.5.10]{JV}).
    Assume that~\ref{prop:loc_eq_same_wek_eq:locrx_wkrx} holds.
    Let $I$ be an invertible ideal with $\OR{I} = O'$.
    Then $I$ is weakly right equivalent to $\OR{I}=O'$ by Proposition~\ref{prop:invOR}.
    Hence $I$ is locally right equivalent to $\OR{I}=O'$.
    By Proposition~\ref{prop:loc_pr_OR}, it follows that $I$ is locally principal.
    The converse follows from Propositions~\ref{prop:wk_eq_cons}.\ref{prop:wk_eq_cons:f} and~\ref{prop:loc_eq_cons}.\ref{prop:loc_eq_cons:f}.

    Assume that~\ref{prop:loc_eq_same_wek_eq:invlocprinc} or~\ref{prop:loc_eq_same_wek_eq:locrx_wkrx} is true for an order $O'$.
    The fact that $\condi$ holds for any inclusion $O \subseteq O'$ follows from Lemma~\ref{lem:pullback_lattices}.
\end{proof}

Note that Example~\ref{ex:kap} exhibits an algebra which does not satisfy Conditions~\ref{prop:loc_eq_same_wek_eq:invlocprinc} and~\ref{prop:loc_eq_same_wek_eq:locrx_wkrx} from the previous proposition.
We don't know if this algebra satisfies condition $\condi$ or not.
In fact, we do not know of an example of an algebra that does not admit~$\condi$.

The next corollary tells us that if $B$ is separable (see the beginning of Section~\ref{sec:duality} for the definition) any inclusion of order $O\subseteq O'$ in $B$ with $O'$ maximal satisfies~$\condi$.
\begin{cor}\label{cor:maxorder}
    Let $B$ be a separable algebra.
    Let $O$ be an order in $B$ and $O'$ a maximal order containing $O$.
    Then $O \subseteq O'$ satisfies $\condi$.
\end{cor}
\begin{proof}
    By~\cite[Thm.~18.10]{Rei03} we know that every one-sided $O'$-ideal is locally principal.
    Therefore we conclude by Proposition~\ref{prop:loc_eq_same_wek_eq}.
\end{proof}

We say that an $R$-lattice $I$ is {\bf left projective} if $I^{-1}I=\OR{I}$ (i.e.~projective as a left $\OL{I}$-module), {\bf right projective} if $II^{-1}=\OL{I}$ (i.e.~projective as a right $\OR{I}$-module), and {\bf projective} if it is left and right projective. See also~\cite[Thm.~20.3.3.(a)]{JV}.

\begin{lemma}\label{lemma:sidedprojinv}
    Let $I$ be a lattice.
    If $I$ right (resp. left) invertible then $I$ is right (resp. left) projective.
\end{lemma}
\begin{proof}
    This is a direct consequence of Lemma~\ref{lemma:uniquesidedinv}.
\end{proof}
It is well known that an $R$-lattice is projective if and only if it is invertible.
See~\cite[Thm.~20.3.3.(b)]{JV}.
On the other hand, the converse to Lemma~\ref{lemma:sidedprojinv} does not hold, as Examples~\ref{ex:sidedinvproj3} and~\ref{ex:sidedinvproj4} show.
To prevent this asymmetry, we introduce the following definition.
\begin{df}\label{df:condii}
    We say that $B$ satisfies $\condii$ if every left projective $R$-lattice $I$ in $B$ is projective.
\end{df}

\begin{lemma}\label{lemma:C2inv}
    Assume that $B$ satisfies $\condii$.
    Then every left projective lattice $I$ is invertible.
\end{lemma}
\begin{proof}
    This follows from the fact that projective is the same as invertible for lattices, as observed above.
\end{proof}

\begin{example}\label{ex:sidedinvproj3}
    Let $R$ be a discrete valuation ring with field of fractions $F$ and uniformizer $\pi$.
    Consider the lattice in $\cM_3(F)$ given by
    \[
        I=
        \begin{pmatrix}
        \pi R & \pi R & R \\
        R & R & \pi^{4} R \\
        R & R & \pi^{2} R
        \end{pmatrix}.
    \]
    We compute
    \begin{align*}
        I^{-1} & =
        \begin{pmatrix}
        \pi^{4}R & R & \pi^{2}R \\
        \pi^{4}R & R & \pi^{2}R \\
        R & \pi R & \pi R
        \end{pmatrix}\\
        \OL{I}=\OR{I^{-1}} & =
        \begin{pmatrix}
        R & \pi R & \pi R \\
        \pi^{4}R & R & \pi^{2}R \\
        \pi^{2}R & R & R
        \end{pmatrix}\\
        I^{-1}I=\OR{I}=\OL{I^{-1}} & =
        \begin{pmatrix}
        R & R & \pi^{4}R \\
        R & R & \pi^{4}R \\
        \pi R & \pi R & R
        \end{pmatrix}\\
        II^{-1} & =
        \begin{pmatrix}
        R & \pi R & \pi R \\
        \pi^{4}R & R & \pi^{2}R \\
        \pi^{2}R & R & \pi^{2}R
        \end{pmatrix}
    \end{align*}
    In particular, we conclude that $I$ is left projective and left invertible, but not right projective, and hence not right invertible.
\end{example}
\begin{example}\label{ex:sidedinvproj4}
    Let $R$ be a discrete valuation ring with field of fractions $F$ and uniformizer $\pi$.
    Consider the lattice in $\cM_4(F)$ given by
    \[ I=
        \begin{pmatrix}
        \pi^{3}R & \pi^{4}R & R & \pi^{2}R \\
        \pi^{8}R & \pi^{4}R & \pi^{5}R & \pi^{-10}R \\
        \pi^{7}R & \pi^{-10}R & \pi^{4}R & \pi^{5}R \\
        \pi^{5}R & R & \pi^{2}R & \pi^{-8}R
        \end{pmatrix}.
    \]
    We compute
    \begin{align*}
        I^{-1} & =
        \begin{pmatrix}
        \pi^{-3}R & \pi^{9}R & \pi^{11}R & \pi^{7}R \\
        \pi^{14}R & \pi^{25}R & \pi^{10}R & \pi^{23}R \\
        R & \pi^{12}R & \pi^{14}R & \pi^{10}R \\
        \pi^{15}R & \pi^{10}R & \pi^{24}R & \pi^{14}R
        \end{pmatrix}\\
        I^{-1}I = \OR{I} =\OL{I^{-1}} & =
        \begin{pmatrix}
        R & \pi^{1}R & \pi^{-3}R & \pi^{-1}R \\
        \pi^{17}R & R & \pi^{14}R & \pi^{15}R \\
        \pi^{3}R & \pi^{4}R & R & \pi^{2}R \\
        \pi^{18}R & \pi^{14}R & \pi^{15}R & R
        \end{pmatrix}\\
        \OL{I} &=
        \begin{pmatrix}
        R & \pi^{12}R & \pi^{14}R & \pi^{10}R \\
        \pi^{5}R & R & \pi^{14}R & \pi^{4}R \\
        \pi^{4}R & \pi^{15}R & R & \pi^{13}R \\
        \pi^{2}R & \pi^{2}R & \pi^{10}R & R
        \end{pmatrix}\\
        \OL{I^{-1}} &=
        \begin{pmatrix}
        R & \pi^{1}R & \pi^{-3}R & \pi^{-1}R \\
        \pi^{17}R & R & \pi^{14}R & \pi^{15}R \\
        \pi^{3}R & \pi^{4}R & R & \pi^{2}R \\
        \pi^{18}R & \pi^{14}R & \pi^{15}R & R
        \end{pmatrix}\\
        II^{-1} &=
        \begin{pmatrix}
        R & \pi^{12}R & \pi^{14}R & \pi^{10}R \\
        \pi^{5}R & R & \pi^{14}R & \pi^{4}R \\
        \pi^{4}R & \pi^{15}R & R & \pi^{13}R \\
        \pi^{2}R & \pi^{2}R & \pi^{10}R & \pi^{6}R
        \end{pmatrix}.
    \end{align*}
    We see that
    \[  \OL{I} \neq \OR{I^{-1}},\quad \OR{I}=\OL{I^{-1}},\quad II^{-1}\neq \OL{I},\quad I^{-1}I= \OR{I}.  \]
    In particular, we conclude that $I$ is left projective, but not left invertible.
    Also $I$ is not right projective, and hence not right invertible.
\end{example}

We conclude the section by describing two large classes of algebras that satisfy both~$\condi$ and~$\condii$.

An $F$-linear map $(\bar{\, \cdot \, }) \colon B\to B$ is called a {\bf standard involution} if $\bar 1 = 1$, $\bar{\bar{\alpha}}=\alpha$, $\bar{\alpha\beta}=\bar\beta\bar\alpha$ and $\alpha\bar\alpha\in F$ for every $\alpha,\beta\in B$.
For an element $\alpha\in B$, its {\bf reduced trace} is defined as $\trd(\alpha)=\alpha+\bar\alpha$ and its {\bf reduced norm} is defined as $\nrd(\alpha)=\alpha\bar\alpha$.

\begin{prop}\label{prop:invC1C2}
    If $B$ is commutative or has a standard involution, then $B$ satisfies $\condi$ and $\condii$.
\end{prop}
\begin{proof}
    Assume that $B$ is commutative.
    Then it is clear that $\condii$ is satisfied.
    It is well known that that the equivalent conditions of Proposition~\ref{prop:loc_eq_same_wek_eq} are satisfied.
    The fact that $B$ satisfies $\condi$ follows from~\cite[Sec.~2]{Wie84} and~\cite[Eq.~4.3.1]{LW85}.

    If $B$ has a standard involution, then any order in $B$ satisfies Proposition~\ref{prop:loc_eq_same_wek_eq}.\ref{prop:loc_eq_same_wek_eq:invlocprinc} by~\cite[Main Thm.~16.6.1]{JV}. By the same proposition, $B$ satisfies $\condi$.
    We will now show that it satisfies $\condii$.
    Let $I$ be a left projective lattice, that is, $I^{-1} I = \OR{I}$.
    We want to show that $I$ is right projective, that is $II^{-1} = \OL{I}$.
    Firstly we note that both equalities $I^{-1}I = \OR{I}$ and $II^{-1} = \OL{I}$ can be checked locally by~\cite[Thm.~9.4.9]{JV}, hence we may assume that $R$ is a discrete valuation ring.
    Secondly, by~\cite[16.6.9]{JV}, there is an element $\alpha \in I$ such that $J = \alpha^{-1} I$ is a semi-order, that is, $1\in J$ and $\nrd{J}=R$.
    By Lemma~\ref{lemma:colonalpha}.\ref{lemma:colonalpha:2} we have that $\OR{J}=\OR{I}$.
    Hence, using Lemmas~\ref{lemma:char_qinv}.\ref{lemma:char_qinv:1} and~\ref{lemma:colonalpha}.\ref{lemma:colonalpha:0}, we see that
    \[
    J^{-1} = \ccL{\OR{J}}{J} = \ccL{\OR{I}}{\alpha^{-1} I} = \ccL{\OR{I}}{I} \alpha = I^{-1} \alpha.
    \]
    Therefore $I$ is right projective if and only if $J$ is right projective. A similar argument holds for left projectivity, noting that $\OL{J} = \alpha^{-1} \OL{I} \alpha$ by Lemma~\ref{lemma:colonalpha}.\ref{lemma:colonalpha:0}.

    We show that $\bar{J} = J$. For $\beta \in J$ we have
    \[
    \overline{\beta} = \trd(\beta) - \beta = \nrd(1 + \beta) - \nrd(\beta) - \beta - 1,
    \]
    which lies in $J$ as $\nrd(J) = R \subseteq J$, hence $\overline{J} \subseteq J$.
    The other inclusion holds by symmetry.
    It follows that $\bar{J^{-1}} = J^{-1}$ as well:
    \[
    \bar{J^{-1}} = \bar{\ccL{\OR{J}}{J}} = \ccR{\bar{\OR{J}}}{\bar{J}} = \ccR{\OL{\bar{J}}} {J} = \ccR{\OL{J}}{J} = J^{-1}.
    \]
    Then
    \[
    \OL{J} = \OL{\bar{J}} = \bar{\OR{J}} = \bar{J^{-1} J} = \bar{J} \cdot \bar{J^{-1}} = J J^{-1},
    \]
    which concludes the proof.
\end{proof}

\section{Duality}
\label{sec:duality}
We briefly state some notions from~\cite[Sec.~7.8 and~7.9]{JV}. As usual, let $R$ be a Dedekind domain with field of fractions $F$. Choose a separable closure $F^{\text{sep}}$ of $F$. In this section we assume our finite-dimensional $F$-algebra $B$ to be {\bf separable}, that is, there exist integers $r_1, \dots, r_n$ and an isomorphism
\[
\phi \colon B \otimes_F F^{\text{sep}} \overset{\sim}{\longrightarrow} \prod_{i = 1}^n \mathcal{M}_{r_i}(F^{\text{sep}}).
\]
For an element $\alpha \in B$ we refine the {\bf reduced trace} $\trd(\alpha)$ to be the sum of the traces of the components of $\phi(\alpha \otimes_F 1)$. It follows from the Skolem-Noether Theorem that this is independent of the choice of the isomorphism $\phi$. Moreover, $\trd(\alpha)$ is always an element of $F$. This definition coincides with the one given in Section~\ref{sec:comparison} when $B$ has a standard involution.

For an $R$-lattice $I$, we define the {\bf (trace) dual} of $I$ as the $R$-lattice
\[ \dual{I}=\set{ x \in B : \trd{xI}\subseteq B }=\set{ x \in B : \trd{Ix}\subseteq B }. \]
In the following lemma we collect several useful properties of the dual.
\begin{lemma}\label{lemma:dual}
    Let $I$ and $J$ be lattices.
\begin{itemize}
    \item $\OR{I}=\OL{\dual{I}}$ and $\OL{I}=\OR{\dual{I}}$.
    \item $\dual{(IJ)} = \ccR{\dual{I}}{J} = \ccL{\dual{J}}{I}$.
    \item $\OL{I} = \dual{(I\dual{I})}$ and $\OR{I}=\dual{(\dual{I}I)}$.
    \item $ I = \dual{(\dual{I})} $.
\end{itemize}
\end{lemma}

\begin{proof}
    See~\cite[Sec.~15.6]{JV}.
\end{proof}

The remainder of this section is devoted to proving Theorem~\ref{thm:wkclfinitequot}, which shows that certain representatives of the weak equivalent classes with prescribed right order $O$ can be found in between two explicit lattices depending only on $O$. The trace dual is a crucial ingredient. We proceed with two technical lemmas.

\begin{lemma}\label{lemma:inO'ff}
    Let $O\subseteq O'$ be orders.
    Let $I$ be an lattice with $O \subseteq \OR{I}$ and $IO'=O'$.
    Let $\frf$ be a lattice with $\frf\subseteq O$ and $O'\subseteq \OL{\frf}$.
    Then
    \[ \frf \subseteq I \subseteq O'. \]
\end{lemma}
\begin{proof}
    Observe that we have
    \[ I = IO \subseteq IO' = O'. \]
    By hypothesis we have $O'\frf = \frf$ and $I\frf \subseteq IO = I$.
    Hence
    \[ I\frf = I O'\frf = O'\frf = \frf. \]
    Putting everything together we obtain
    \[ \frf = I\frf \subseteq I \subseteq O'. \qedhere\]
\end{proof}

\begin{lemma}\label{lemma:ext_rx_wk_eq}
    Let $O\subseteq O'$ be orders.
    Assume $\dual{O}O'$ is weakly right equivalent to $O'$.
    Let $I$ be a lattice with $\OR{I}=O$.
    Then the extension $IO'$ is left projective and satisfies $\OR{IO'} = O'$.

    If we assume additionally that $B$ satisfies $\condii$ then $IO'$ is weakly right equivalent to $O'$.
\end{lemma}
\begin{proof}
    By definition, there exists an invertible lattice $L$ such that
    \[ L \dual{O}O' = O'. \]
    Proposition~\ref{prop:wk_eq_cons}.\ref{prop:wk_eq_cons:a} implies that
    \[\OR{\dual{O}O'}=\OL{O'}=O'=\OR{O'}.\]
    By Lemma~\ref{lemma:dual} we have $\dual{O} = \dual{\OR{I}} = \dual{I}I$, hence
    $ L\dual{I}IO' = O'$. It follows that $\OR{IO'} = O'$, as
    \[
    O' = \OR{O'} \subseteq \OR{IO'} \subseteq \OR{L\dual{I}IO'} = \OR{O'} = O'.
    \]
    We show that $(IO')^{-1}IO' = \OR{IO'}$. The left hand side is contained in the right hand side by Lemma~\ref{lemma:char_qinv}.\ref{lemma:char_qinv:1}. Combining $\dual{O} = \dual{\OR{I}} = \dual{I}I$ and the definition of $L$ gives
    \[
    (IO') L\dual{I} (IO') = IO' L(\dual{I} I)O' = IO' (L \dual{O} O') = IO'O' = IO',
    \]
    hence $L\dual{I} \subseteq (IO')^{-1}$.
    Thus
    \[
    \OR{IO'} = O' = L\dual{I}IO' \subseteq (IO')^{-1}IO'.
    \]
    We conclude that $(IO')^{-1}IO' = \OR{IO'}$, that is, $IO'$ left projective.

    If we assume $B$ satisfies $\condii$ then $IO'$ is invertible by Lemma~\ref{lemma:C2inv}.
    In particular, $IO'$ is weakly right equivalent to $\OR{IO'} = O'$ by Proposition~\ref{prop:invOR}.
\end{proof}

\begin{thm}\label{thm:wkclfinitequot}
    Assume that $B$ satisfies $\condi$ and $\condii$.
    Let $O\subseteq O'$ be orders such that $\dual{O}O'$ is weakly right equivalent to $O'$.
    Let $\frf$ be a lattice with $\frf\subseteq O$ and $O'\subseteq \OL{\frf}$.
    Then for every lattice $I'$ with $\OR{I'}=O$ there exists a lattice $I$,  weakly right equivalent to $I'$, such that $IO'=O'$ and
    \[ \frf \subseteq I \subseteq O'. \]
\end{thm}

\begin{proof}
    Since $\condii$ holds, by Lemma~\ref{lemma:ext_rx_wk_eq}, we have that $I'O'$ is weakly right equivalent to $O'$.
    By Proposition~\ref{prop:wk_eq_cons}.\ref{prop:wk_eq_cons:f} there exists an invertible lattice $L'$ with $\OR{L'} = \OL{I'O'}$ such that
    \[ L'I'O' = O'. \]
    As $B$ satisfies $\condi$ and $\OL{I'} \subseteq \OL{I'O'}$, there exists a lattice $L$ with $\OR{L} = \OL{I'}$ such that $L\OL{I'O'} = L'$.
    Let $I=LI'$; note that $I$ is weakly right equivalent to $I'$ and satisfies $IO'=O'$ as
    \[
        IO' = LI'O' = L \OL{I' O'} I' O' = L' I' O' = O'.
    \]
    Moreover, $O = \OR{I'} \subseteq \OR{L I'} = \OR{I}$. Hence by Lemma~\ref{lemma:inO'ff} we conclude
    \[ \frf \subseteq I \subseteq O'. \qedhere \]
\end{proof}

\begin{remark}\label{rmk:existenceO'frf}
    As we explain now, there exist $O'$ and $\frf$ meeting the hypotheses of the theorem.

    As $B$ is a separable algebra, every order $O$ is contained in a maximal order $O'$. Observe that $\OR{\dual{O}O'} = O'$ by maximality and therefore $\dual{O}O'$ is a right $O'$-ideal, hence locally principal by
   ~\cite[Thm.~18.10]{Rei03}
    hence invertible. Therefore $\dual{O}O'$ is weakly right equivalent to $O'$ by Proposition~\ref{prop:invOR}.

    Observe also that $\frf = \ccR{O}{O'}$ meets the hypothesis of the theorem, since $\ccR{O}{O'} \subseteq O$ because $1 \in O'$, and $\ccR{O}{O'}$ is a left $O'$-ideal by Lemma~\ref{lemma:colonalpha}.\ref{lemma:colonalpha:7}.
\end{remark}

Provided that the residue fields of the maximal ideals of $R$ are all finite, this theorem implies that the set of weak right equivalence classes is finite, which could be considered a `weak' version of the Jordan-Zassenhaus Theorem \cite[Thm~26.4]{Rei03}, see the Corollary~\ref{cor:wkJZ} below.
A similar statement holds for isomorphism classes of orders, where we say that two orders $O$ and $O'$ are isomorphic if there exists $\alpha \in B^\times$ such that $O' = \alpha O \alpha^{-1}$.

\begin{cor}\label{cor:wkJZ}
    Assume that $B$ satisfies $\condi$ and $\condii$.
    Let $\mathcal{J}$ be a set of representatives $J$ of the weak right equivalence classes satisfying $\OR{J}=O$. Assume that every maximal ideal $\p$ of $R$ has finite residue field $R/\p$. Then $\mathcal{J}$ is finite.
\end{cor}
\begin{proof}
    By Theorem~\ref{thm:wkclfinitequot}, all weak equivalence classes with given right order $O$ have a representative $J$ contained in some order $O'$ containing $O$ and contain some lattice $\frf$ contained in $O'$.
    By Lemma~\ref{lem:finmanylatt}, there are only finitely many such $J$, hence $\mathcal{J}$ is finite.
\end{proof}

\section{Algorithmic implementation}\label{sec:algorithms}
Let $R$ be a Dedekind domain satisfying Lemma~\ref{lem:finmanylatt}.\ref{lem:finmanylatt:1} with fraction field $F$, and let $B$ be a separable finite-dimensional $F$-algebra that satisfies $\condi$ and $\condii$.
In this section we explain how to compute the right equivalence classes of lattices with prescribed right order $O$ in $B$.
The implementation in Magma~\cite{Magma} is available at
 \url{https://github.com/harryjustussmit/IdlClQuat}.
The algorithm is based on Theorem~\ref{thm:righteqclasses}: we can combine the invertible right equivalence classes (with appropriate right orders) and the weak right equivalence classes (with right order $O$) in order to find a list of representatives of the right equivalence classes.
We will assume that the former consists of a finite list, and that we have a
method, which we call \texttt{InvertibleRightEquivalenceClasses}, to compute it.
An example of this that works for every order in a definite quaternion algebra over a totally real number field is described in \cite[Alg.~2]{KirschLorch16}, while a method that works for every order in a quaternion algebra over a number field can be found in
Appendix \ref{appendix:invidls}.

For the latter we describe Algorithm~\ref{alg:rightequivalenceclasses}.
It makes use of Theorem~\ref{thm:wkclfinitequot}, which ensures all representatives lie between an overorder $O'$ of $O$ and a specific lattice $\frf \subset O'$.
To find this overorder, we invoke another black box called \texttt{Overorders}, which computes all overorders of $O$, after which we go through the list to find a suitable one.
Under the equivalent hypotheses of Lemma~\ref{lem:finmanylatt}, there exist finitely many lattices between $\frf$ and $O'$.
In particular, in our search for $O'$ we want to keep the quotient $O'/\frf$ as small as possible.
We enumerate these finitely many lattices, divide them into weak right equivalence classes, and choose representatives.

Note that, for every inclusion of orders $O\subseteq O'$ we have
\[ O \subseteq O' \subseteq \dual{O'} \subseteq \dual{O}. \]
Hence, under the equivalent hypotheses of Lemma~\ref{lem:finmanylatt},
computing the overorders of $O$ is a finite problem.
For an efficient implementation of \texttt{Overorders} see~\cite[Sec.~4]{HofSirc}.
Note that, in many cases, the orders containing $O$ are well understood by the work of Brzezinski, see for example \cite{Brezi83}.

Moreover, we assume that we have algorithms to compute the sum, product, and colon of two lattices.
Finally, given two orders $O\subset O'$, we need an algorithm to enumerate right $O$-ideals $I$ such that $IO'=O'$ and $\frf\subseteq I \subseteq O'$, where $\frf=\ccR{O}{O'}$, cf.~Theorem~\ref{thm:wkclfinitequot}.
This can be done using the following procedure, which is a modification of the method described in~\cite[Sec.~5.2]{FHS19}.
Put $G=O'/\frf$.
The right $O$-ideals $I$ we want to list are in bijection with the sub-$O$-modules $H$ of $G$ such that $HO'=G$,
which can be enumerated by recursively searching for maximal sub-$O$-modules $M$ of $G$ which satisfy $MO'= G$.
Note that if $MO' \subsetneq G$ then we can exclude it from the recursion, because all the submodules $M' \subseteq M$ will also satisfy $M'O' \subsetneq G$.
Moreover, if $M$ is maximal in $G$ then there exists a unique prime ideal $\p$ of $R$ such that $ \Ann_R(G/M) = \p$.
Hence $M$ corresponds to a sub-$R/\p$-vector space of $G/\p G$ which is stable under the induced action of $O$.
These vector spaces can be efficiently enumerated using \cite{Par84} and \cite[Sec.~7.4]{HoltEickOBrien05}.

\begin{algorithm}[h]
    \vspace{0.07cm}
 \SetKwData{Proj}{proj} \SetKwData{Com}{com}
 \KwIn{Two lattices $I$ and $J$.}
 \KwOut{Whether or not $I$ and $J$ are weakly right equivalent.}
    $C_1:=\ccL{I}{J}$\;
    $C_2:=\ccL{J}{I}$\;
    \Proj$:= (C_1C_2=\OL{C_1})$ {\bf and} $(C_2C_1=\OR{C_1})$\;
    \Com$:= (\OL{C_1}=\OR{C_2})$ {\bf and} $(\OR{C_1}=\OL{C_2})$\;
    \KwRet{$(\Proj$ \bf{and} $\Com)$}\;
 \caption{\label{alg:isweaklyequivalent}\texttt{IsWeaklyRightEquivalent}}
\end{algorithm}

Algorithm~\ref{alg:isweaklyequivalent} is correct by Theorem~\ref{thm:wk_eq}.


\begin{algorithm}[h]
    \vspace{0.07cm}
\SetKwFunction{Overorders}{Overorders} \SetKwData{Representatives}{Representatives} \SetKwFunction{IsWeaklyRightEquivalent}{IsWeaklyRightEquivalent}
 \KwIn{An order $O$.}
 \KwOut{Representatives of the weak right equivalence classes with right order equal to $O$.}
 \Representatives $:= \{\, \}$\;
 Choose $O'$ in \text{\Overorders}(O) such that $\IsWeaklyRightEquivalent(\dual{O}O', O')$\;
 $\frf := \ccR{O}{O'}$\;
 $G := O'/\frf$; \tcp*[f]{A finite right $O$-module.}\\
 \For{{\bf every right sub-$O$-module $H$ of $G$ such that $HO'=G$}}
 {
    Choose a lift $\ell \subseteq O'$ of a generating set of $H$\;
    Let $I$ be the right $O$-ideal generated by $\ell$ and $\frf$\;
    \If{$(\OR{I} = O)$}
    {
    \If{$((${\bf not} \IsWeaklyRightEquivalent{$I$, $J$}$)$ {\bf for all $J$ in} $\text{\Representatives})$}
        {
            Add $I$ to \Representatives\;
        }
    }

 }
\KwRet{\Representatives}\;
 \caption{\label{alg:weakrightequivalenceclasses} \texttt{WeakRightEquivalenceClasses}}
\end{algorithm}

Algorithm~\ref{alg:weakrightequivalenceclasses} is correct by Theorem~\ref{thm:wkclfinitequot} and Remark~\ref{rmk:existenceO'frf}.
The running time of Algorithm~\ref{alg:weakrightequivalenceclasses} is determined by the number of the right sub-$O$-modules of $O'/\frf$. In order to improve the efficiency one would like to minimize this number, which is unfortunately hard to estimate.
Clearly, if for two orders $O'\subset O''$ we have that $\dual{O}O'$ is weakly right equivalent to $O'$ and $\dual{O}O''$ is weakly right equivalent to $O''$ then running \texttt{WeakRightEquivalenceClasses} with $O'$ will be faster than running it with $O''$.

\begin{algorithm}[h]
    \vspace{0.07cm}
    \SetKwFunction{InvertibleRightEquivalenceClasses}{InvertibleRightEquivalenceClasses}
    \SetKwFunction{WeakRightEquivalenceClasses}{WeakRightEquivalenceClasses}
    \SetKwData{Representatives}{Representatives}
 \KwIn{An order $O$.}
 \KwOut{Representatives of the right equivalence classes with right order equal to $O$.}
 $\Representatives := \{\, \}$\;
 \For{$J$ \bf{in} $\WeakRightEquivalenceClasses(O)$}
 {
    \For{$L$ \bf{in} $\InvertibleRightEquivalenceClasses(\OL{J})$}
    {
        Add $LJ$ to \Representatives\;
    }
 }
 \KwRet{\Representatives}\;
 \caption{\label{alg:rightequivalenceclasses} \texttt{RightEquivalenceClasses}}
\end{algorithm}

The correctness of Algorithm~\ref{alg:rightequivalenceclasses} is Theorem~\ref{thm:righteqclasses}. In order to compute the equivalence classes of all right $O$-ideals instead of just those with right order equal to $O$, one can simply loop over all overorders of $O$ and apply Algorithm~\ref{alg:rightequivalenceclasses} for each of these overorders.

\section{Brandt Matrices}\label{sec:brandtmatrices}

The equivalence classes of invertible lattices can be used to compute modular forms through the use of Brandt matrices, see \cite{Pizer80} and \cite[Ch.~41]{JV}.
In the following we extend the definition of a Brandt matrix to include both the invertible and non-invertible equivalence classes.

Let $O$ be an order in a definite quaternion algebra $B$ over $\Q$.
First, we extend the usual definition of index of two lattices $I$ and $J$ by setting
\[ [I:J]=\frac{[I:I\cap J]}{[J:I \cap J]}. \]
Observe that for every $\alpha\in B$ we have
\[ [I : \alpha J] = [I:J]\nrd(\alpha)^2 = [I:J]N_{B/\Q}(\alpha), \]
where $N_{B/\Q}$ is the algebra $\Q$-norm of $B$.
Let $I_1,\ldots,I_r$ be a set of right $O$-ideals representing the right equivalence classes of lattices with $\OR{I_i}=O$ for $i=1,\ldots,r$.
For any positive integer $n$ we define an $r\times r$ matrix $T(n)$ in the following way. For $1\leq i,j \leq r$ put
\begin{equation*}\label{eq:brandt_matrix}
    T(n)_{i,j} = \#\set{ J \subseteq I_j : [I_j:J]=n^2 \text{ and }
    J = \alpha I_i \text{ for some }\alpha\in B^\times }.
\end{equation*}
Let $J$ be as in the definition of $T(n)_{i,j} $.
Then $\alpha\in \ccL{J}{I_i}\subseteq \ccL{I_j}{I_i}$ and
%
\[ [I_j:J] = [I_j:I_i]\nrd(\alpha)^2=[I_j:I_i]N_{B/\Q}(\alpha)=n^2. \]
Conversely, for $\alpha \in \ccL{I_j}{I_i}$ satisfying
\[ [I_j:I_i]\nrd(\alpha)^2=n^2, \]
we obtain a lattice $J=\alpha I_i$ satisfying the conditions.
Moreover, we have that $\alpha I_i = \alpha' I_i$ for some $\alpha'$ if and only if $\alpha^{-1}\alpha'$ is in $\OL{I_i}^\times$.
Therefore, defining
\begin{align*}
    \cQ_{i,j}: \ccL{I_j}{I_i} &\longrightarrow \Z\\
                \alpha  &\longmapsto \nrd(\alpha)\sqrt{[I_j:I_i]},
\end{align*}
we find that
\[ T(n)_{i,j} = \frac{1}{\# \OL{I_i}^\times } \# \set{ \alpha \in \ccL{I_j}{I_i} \ :\ \cQ_{i,j}(\alpha) = n }.\]
If we fix any $\Z$-basis of $\ccL{I_j}{I_i}$, we see that $\cQ_{i,j}$ gives rise to a positive definite quadratic form $\widetilde \cQ_{i,j}$.
Hence, to compute $T(n)$, we can use algorithms to enumerate $\Z$-vectors $\vec{x}$ such that $\widetilde\cQ_{i,j}(\vec{x}) = n$.

The series
\[ \Theta_{i,j}(q) = \sum_{n=0}^\infty T(n)_{i,j} q^n \]
is a scalar multiple of the theta series of the quadratic form $\widetilde\cQ_{i,j}$, and therefore a modular form.

\begin{example}
    Consider the quaternion algebra $B\langle i,j,k \rangle = \left( \frac{-1,-3}{\Q} \right)$
    where $i^2=-1$, $j^2=-3$, $ij=k$.
    The order
    $ O = \Z + 2 i\Z + 2j\Z + 2k\Z $ has $2$ right weak equivalences classes with right order $O$.
    Both these classes have left orders with exactly $4$ invertible right ideals.
    It follows that there are $4$ right equivalence classes with right order $O$.
    The Brandt matrices~$T(n)$ for~$n=1,\dots,14$ together with the first coefficients of the corresponding theta series are listed at~\url{https://github.com/harryjustussmit/IdlClQuat/blob/main/examples/example_brandt_matr.m}
    The code to produce and extend this data is available at~\url{https://github.com/harryjustussmit/IdlClQuat/blob/main/examples/extra_data_Ex_7.1.pdf}
\end{example}

\appendix

\section{Computing invertible ideals (by John Voight)}
\label{appendix:invidls}
In this appendix, we exhibit an algorithm to compute a set of representatives for the right class set of a quaternion order over a number ring.  This generalizes work of Kirschmer--Voight \cite{KirschmerVoight10} from the case of an Eichler order to a general quaternion order.  

Let $F$ be a number field of degree $n \colonequals [F:\Q]$ and let $R$ be its ring of integers.  Let $B$ be a quaternion algebra over $F$ and let $O \subset B$ be an $R$-order.  Let $\Cls O$ be the set of equivalence classes of invertible (equivalently, locally principal) right $O$-ideals \cite[Chapter 17]{JV} under the usual equivalence relation $I \sim J$ if there exists $\alpha \in B^\times$ such that $J=\alpha I$.  

First, suppose that $B$ is indefinite (i.e., either $F$ has a complex place or $B$ has a split real place).  Then \cite[Theorem 28.5.5]{JV} there is a bijection from the set of right ideal classes of $\calO$ to a class group of $R$.  More precisely, let $\Omega$ be the set of real places of $F$ that ramify in $B$ and let $F_{>_\Omega 0}^\times \leq F^\times$ be the subgroup of elements that are positive at all places in $\Omega$.  Then the reduced norm induces a bijection
\[ \nrd \colon \Cls O \xrightarrow{\sim} \Cl_{G(O)} R \]
where 
\[ \Cl_{G(O)} R \colonequals F_{>_{\Omega} 0}^\times \backslash \widehat{F}^\times / \nrd(\widehat{O}^\times). \]

This class group is effectively computable using methods of computational class field theory \cite[Remark 6.12]{SmertnigVoight19}: in particular, we compute a set of representatives of $\Cl_{G(O)} R$ consisting of prime ideals $\frakp$ not dividing the reduced discriminant $\discrd \calO$.  Then to enumerate $\Cls O$, we need only exhibit an invertible 
ideal of prime norm $\frakp$.  For this, we compute an embedding $\iota_\p \colon O \hookrightarrow \mathrm{M}_2(R_\p)$ \cite[Corollary 7.13]{Voight13} and take $I=\p O+\beta O$ where 
\[ \iota_\p(\beta) \equiv \begin{pmatrix} 1 & 0 \\ 0 & \pi \end{pmatrix} \pmod{\frakp^2} \]
and $\pi \in R$ is a uniformizer for $\frakp$.  In fact, a quick local calculation shows we may take $\beta \in \calO$ to be any element with $\ord_\frakp(\nrd(\beta))=1$.  

\begin{remark}
For further improvements, as well as the solution to the principal ideal problem (exhibiting a generator of a right ideal if it exists), see Kirschmer--Voight \cite[section 4]{KirschmerVoight10} and  Page \cite{Page14}.  
\end{remark}

We are left with the case where $B$ is definite, in particular $B$ is a division algebra.  In this case, we may use one of two possible strategies.  

First, we simply adjust the $\frakp$-neighbor algorithm given by Kirschmer--Voight \cite[Algorithm 7.4]{KirschmerVoight10}, replacing the class group with $\Cl_{G(O)} R$ in Step (1).  We give a concise description here for convenience, referring the reader to the paper for algorithmic details.  As above, we compute a set $S$ of representatives $\frakp \nmid \discrd \calO$ for $\Cl_{G(O)} R$.  Starting with $[\mathcal{O}]$, we iteratively compute the set of $\frakp$-neighbors (those $J \subseteq I$ with $[I:J]_R = \frakp^2$) for $\frakp \in S$, testing for equivalence using a short vector calculation, until no new classes are found.  This algorithm terminates with correct output by an application of the theorem of strong approximation.

A second approach is a modification of the above, using the mass formula.  Define the {\bf mass} of $\Cls \calO$ as a weighted class number
\[ \mass(\Cls \calO) \colonequals \sum_{[I] \in \Cls O} [\OL{I}^\times : R^\times]^{-1}. \]
Then the Eichler mass formula \cite[Main Theorem 26.1.5]{JV} gives
\begin{equation} \label{eqn:massclsO}
\mass(\Cls \calO) = \frac{2\zeta_F(2)}{(2\pi)^{2n}} d_F^{3/2} h_F  \mathsf{N}(\frakN) \prod_{\frakp \mid \frakN} \lambda(\calO,\frakp) 
\end{equation}
where:
\begin{itemize}
\item $d_F$ is the absolute discriminant of $F$,
\item $h_F$ is the class number of $F$, 
\item $\frakN \colonequals \discrd(\calO)$,
\item $\mathsf{N}$ is the absolute norm \cite[16.4.8]{JV}, and 
\item $\lambda(\calO,\frakp)$ are explicitly given local factors \cite[(26.1.2)]{JV}.  
\end{itemize}
The expression for the mass \eqref{eqn:massclsO} is effectively computable: in fact, see Kirschmer--Voight \cite[section 5]{KirschmerVoight10} for an algorithm (with running time estimate) to compute $\zeta_F(2)d_F^{3/2}/(2\pi)^{2n} \in \Q$.  

We then proceed as in Kirschmer--Voight \cite[Remark 7.5]{KirschmerVoight10}: we just enumerate ideals $\fraka \subseteq R$ coprime to $\discrd \calO$ and all right ideals of reduced norm $\fraka$, testing for equivalence class again using a short vector calculation, and we continue until the sum $\sum_{[I]} [\OL{I}^\times : R^\times]^{-1}$ over the computed set of representatives $[I]$ is equal to $\mass(\Cls \calO)$.  An implementation of this algorithm by Smertnig--Voight (used in the enumeration of definite orders with locally free cancellation \cite{SmertnigVoight19}) in \textsf{Magma} is available at \url{https://github.com/dansme/hermite}.  

\begin{remark}
It is possible to provide a rigorous estimate on the running time for this algorithm, as in the case of Eichler orders \cite[Theorem 7.9]{KirschmerVoight10}.
\end{remark}

An amalgam of these two approaches (using $\frakp$-neighbors, but with an early abort using the mass formula) is implemented in \textsf{Magma}.

\begin{remark}
More naive, enumerative algorithms are also possible.  We give a very brief indication to conclude.

We recall \cite[Proposition 17.7.19]{JV} that there exists an effectively computable constant $C>0$ such that every class in $\Cls O$ admits a representative $I$ with 
$\mathsf{N}(I)\leq C$.  So to compute $\Cls O$, we first compute the finite set of invertible right $O$-ideals $I\subseteq O$ with bounded (absolute) norm and second organize them according to their right equivalence class (using short vector algorithms).  

For the first step, we first note that for $I$ invertible we have $\mathsf{N}(I)=\mathsf{N}(\nrd{I})^2$, so we can loop over those ideals $\fra \subseteq R$ such that $\mathsf{N}(\fra) \leq \sqrt{C}$ by factoring in $R$.  A right ideal $I \subseteq \calO$ with $\nrd(I)=\fraka$ is invertible if and only if $I = \fra O + \beta O$ with $\ord_\frakp(\nrd(\beta))=\ord_\frakp(\fraka)$ for all primes $\frakp \mid \fraka$ (a slight extension of \cite[Exercise 16.6]{JV}), so we reduce to enumerating elements $\overline{\beta} \in \calO/\fraka \calO$ such that there exists a lift $\beta \in \calO$ satisfying the norm condition, which can be checked linearly since
\[ 
\nrd(\beta+c\alpha) = \nrd(\beta)+c\trd(\alpha\overline{\beta})+c^2\nrd(\alpha)  \]
for $c \in R$ and $\alpha \in \calO$.  
\end{remark}

\subsection*{Acknowledgements} Voight was supported by a Simons Collaboration Grant (550029).

\providecommand{\bysame}{\leavevmode\hbox to3em{\hrulefill}\thinspace}
\providecommand{\MR}{\relax\ifhmode\unskip\space\fi MR }
\providecommand{\MRhref}[2]{%
  \href{http://www.ams.org/mathscinet-getitem?mr=#1}{#2}
}
\providecommand{\href}[2]{#2}

\end{document}